\documentclass[11pt]{amsart}
\usepackage{amsmath, amsthm, amssymb,amscd}
\usepackage{float}
\usepackage{graphicx}
\usepackage{xypic}
\usepackage[arrow, matrix, curve]{xy}
\usepackage{color}
\usepackage{enumitem}

\usepackage[english]{babel}
\newtheorem{thm}{Theorem}[section]
\newtheorem*{thm-nl}{Theorem}
\newtheorem*{prop-nl}{Proposition}

\newtheorem{lem}[thm]{Lemma}

\def\PP{{\textbf P}}
\def\OO{\mathcal{O}}

\def\cM{\mathcal{M}}
\def\cR{\mathcal{R}}

\def\rr{\overline{\mathcal{R}}}

\def\Pic0{{\rm Pic}^0(X)}

\newtheorem*{cor-nl}{Corollary}

\newtheorem*{conjecture-nl}{Conjecture}
\newtheorem*{quest-nl}{Question}
\newtheorem*{quests-nl}{Questions}

\newtheorem{prop}[thm]{Proposition}

\theoremstyle{remark}

\title{{The resolution of paracanonical curves of odd genus}}

\author[G. Farkas]{Gavril Farkas}
\address{Humboldt-Universit\"at zu Berlin, Institut f\"ur Mathematik,  Unter den Linden 6
\hfill \newline\texttt{}
 \indent 10099 Berlin, Germany} \email{{\tt farkas@math.hu-berlin.de}}

\author[M. Kemeny]{Michael Kemeny}

\address{Stanford University, Department of Mathematics, 450 Serra Mall
\hfill \newline\texttt{}
 \indent CA 94305, USA} \email{{\tt michael.kemeny@gmail.com}}

\begin{document}

\begin{abstract}
We prove the Prym--Green conjecture on minimal free resolutions of paracanonical curves of odd genus. The proof proceeds via
curves lying on ruled surfaces over an elliptic curve.

\end{abstract}

\maketitle
\setcounter{section}{-1}
\section{Introduction}
The study of torsion points on Jacobians of algebraic curves has a long history in algebraic geometry and number theory. On the one hand, torsion points of Jacobians have been used to rigidify moduli problems for curves, on the other hand, such a torsion point determines an unramified cyclic cover over the curve in question, which gives rise to a (generalized) Prym variety, see \cite{BL} Chapter 12 for an introduction to this circle of ideas.

\vskip 3pt

Pairs  $[C,\tau]$, where $C$ is a smooth curve of genus $g\geq 2$ and $\tau \in \text{Pic}^0(C)$ is a non-trivial torsion line bundle of order $\ell\geq 2$ form an irreducible moduli space $\cR_{g,\ell}$. One may view this moduli space as a higher genus analogue of the level $\ell$ modular curve $X_1(\ell)$. There is a finite cover $$\cR_{g,\ell}\rightarrow \cM_g$$ given by forgetting the $\ell$-torsion point. Following ideas going back to Mumford, Tyurin and many others, linearizing the Abel--Prym embedding of the curve in its Prym variety leads to the study of the properties of  $[C, \tau]$ in terms of the projective geometry of the level $\ell$ \emph{paracanonical} curve $$\varphi_{K_C\otimes \tau}:C\hookrightarrow \PP^{g-2}$$
induced by the line bundle $K_C\otimes \tau$. In practice, this amounts to a qualitative study of the equations and the syzygies of the paracanonical curve in question. For instance, in the case $\ell=2$, there is a close relationship between the study of these syzygies and the Prym map
$$\cR_{g,2} \to \mathcal{A}_{g-1}$$
to the moduli space of principally polarized abelian varieties of dimension $g-1$, which has been exploited fruitfully for some time, see for instance \cite{beauville}. For higher level, the study of these syzygies has significant applications to the study of the birational geometry of $\cR_{g,\ell}$, see \cite{CEFS}.

\vskip 4pt

Denoting by $\Gamma_C(K_C\otimes \tau):=\bigoplus_{q\geq 0} H^0\Bigl(C, (K_C\otimes \tau)^{\otimes q}\Bigr)$ the homogeneous coordinate ring of the paracanonical curve, for integers $p,q\geq 0$, let  $$K_{p,q}(C,K_C\otimes \tau):=\mbox{Tor}^p\Bigl(\Gamma_C(K_C\otimes \tau),\mathbb C\Bigr)_{p+q}$$
be  the Koszul cohomology group of $p$-th syzygies of weight $q$ of the paracanonical curve and one denotes by $b_{p,q}:=\mbox{dim } K_{p,q}(C,K_C\otimes \tau)$ the corresponding Betti number.

\vskip 5pt

The \emph{Prym-Green Conjecture} formulated in \cite{CEFS} predicts that the minimal free resolution of the paracanonical curve corresponding to a general level $\ell$ curve  $[C,\tau]\in \cR_{g,\ell}$ of genus  $g\geq 5$ is \emph{natural}, that is, in each diagonal of its Betti table, at most one entry is non-zero. The naturality of the resolution amounts to the vanishing statements $b_{p,2}\cdot b_{p+1,1}=0$, for all $p$. As explained in \cite{CEFS}, for odd genus $g=2n+1$ this is equivalent to the  vanishing statements
\begin{equation}\label{pgeq}
 K_{n-1,1}(C, K_C\otimes \tau)=0 \; \; \text{and} \; \ \ K_{n-3,2}(C, K_C\otimes \tau)=0.
\end{equation}
Since the differences $b_{p,2}-b_{p+1,1}$ are known, naturality entirely determines the resolution of the general level $\ell$ paracanonical curves and shows that its Betti numbers are as small as the geometry (that is, the Hilbert function) allows.  We refer to \cite{FaLu} and \cite{CEFS} for background on this conjecture and its important implications on the global geometry of $\cR_{g,\ell}$.

In particular, a positive solution to the Prym-Green Conjecture for bounded genus $g<23$ has been shown to be instrumental in determining the Kodaira dimension of $\cR_{g,\ell}$ for small values of $\ell$. The Prym--Green Conjecture is obviously inspired by the classical Green's Conjecture for syzygies of canonical curves stating that the minimal resolution of a general canonical curve $C\subseteq \PP^{g-1}$ is natural. The main result of this paper is a complete solution to this conjecture in odd genus:
\begin{thm}\label{pgmain}
The Prym-Green Conjecture holds for any odd genus $g$ and any level $\ell$.
\end{thm}

Theorem \ref{pgmain} implies that the general level $\ell$ paracanonical curve of genus $g=2n+1\geq 5$ has the following minimal resolution:
\begin{table}[htp!]
\begin{center}
\begin{tabular}{|c|c|c|c|c|c|c|c|c|}
\hline
$1$ & $2$ & $\ldots$ & $n-3$ & $n-2$ & $n-1$ & $n$  & $\ldots$ & $2n-2$\\
\hline
$b_{1,1}$ & $b_{2,1}$ & $\ldots$ & $b_{n-3,1}$ & $b_{n-2,1}$ & 0 & 0 &  $\ldots$ & 0 \\
\hline
$0$ &  $0$ & $\ldots$ & $0$ & $b_{n-2,2}$ & $b_{n-1,2}$ & $b_{n,2}$ & $\ldots$ & $b_{2n-2,2}$\\
\hline
\end{tabular}
\end{center}

\end{table}

where,
 $$ b_{p,1}=\frac{p(2n-2p-3)}{2n-1}{2n\choose p+1} \ \  \mbox{ if } p\leq n-2,\  \mbox{ } \ \  b_{p,2}=\frac{(p+1)(2p-2n+5)}{2n-1}{2n\choose p+2} \  \ \  \mbox{ if } p\geq n-2.$$

In odd genus, the conjecture has been established before for level $2$ in \cite{generic-secant} (using Nikulin surfaces) and for high level $\ell\geq \sqrt{\frac{g+2}{2}}$ in \cite{high-level}
(using Barth--Verra surfaces). Theorem \ref{pgmain} therefore removes any restriction on the level $\ell$. Apart from that, we feel that the rational elliptic surfaces used in this paper are substantially simpler objects than the $K3$ surfaces used in \cite{generic-secant} and \cite{high-level} and should have further applications to syzygy problems. The Prym--Green Conjecture in even genus, amounting to the single vanishing statement
\begin{equation}\label{pgeven}
K_{\frac{g}{2}-2,1}(C,K_C\otimes \tau)=0,
\end{equation}
(or equivalently, $K_{\frac{g}{2}-3,2}(C,K_C\otimes \tau)=0$) is still mysterious. It is expected to hold for any genus and level $\ell>2$. For level $2$, it has been shown to fail in genus $8$ in \cite{CFVV}; a \emph{Macaulay} calculation carried out in \cite{CEFS} indicates that the conjecture  very likely fails in genus $16$ as well.  This strongly suggests  that for level $2$ the Prym--Green Conjecture fails for general \emph{Prym canonical} curves of genera having high divisibility properties by $2$ and in these cases there should be genuinely new methods of constructing syzygies. At the moment the vanishing (\ref{pgeven}) is not even known to hold for arbitrary even genus $g$ in the case when $\tau$ is a general line bundle in $\mbox{Pic}^0(C)$.

\vskip 4pt

By semicontinuity and the irreducibility of $\cR_{g,\ell}$, it is enough to establish the vanishing (\ref{pgeq}) for one particular example of a paracanonical curve of odd genus. In our previous partial results on the Prym--Green Conjecture, we constructed suitable examples $[C,\tau]$ in terms of curves lying on various kinds of lattice polarized $K3$ surfaces, namely the Nikulin and Barth--Verra surfaces. In each case, the challenge lies in realizing the $\ell$-torsion bundle $\tau$ as the restriction of a line bundle on the surface, so that the geometry of the surface can be used to prove the vanishing of the corresponding Koszul cohomology groups, while making sure that the curve $C$ in question remains general, for instance, from the point of view of Brill-Noether theory.  In contrast, in this paper we use the elliptic ruled surfaces recently introduced in \cite{farkas-tarasca} (closely related to the very interesting earlier work of Treibich \cite{treibich}), in order to provide explicit examples of pointed Brill-Noether general curves defined over $\mathbb Q$. These surfaces also arise when one degenerates a projectively embedded $K3$ surface to a surface with isolated, elliptic singularities. They have been studied in detail by Arbarello, Bruno and Sernesi in their important work \cite{ABS} on the classification of curves lying on $K3$ surfaces in terms of their Wahl map.

\vskip 3pt

Whereas our previous results required a different $K3$ surface for each torsion order $\ell$ for which the construction worked, in the current paper we deal with \emph{all} orders $\ell$ using a single surface. This is possible because on the elliptic ruled surface in question, a general genus $g$ curve admits a canonical degeneration within its linear system to a singular curve consisting of a curve of genus $g-1$ and an elliptic tail. This leads to an inductive structure involving curves of every genus and makes possible inductive arguments, while working on the same surface all along.

\vskip 4pt

We introduce the elliptic ruled surface central to this paper. For an elliptic curve $E$, we set
$$\phi:X:=\PP(\mathcal{O}_E\oplus \eta)\rightarrow E,$$ where $\eta \in \text{Pic}^0(E)$ is neither trivial nor torsion. We fix an origin $a \in E$ and let $b\in E$ be such that $\eta=\OO_E(a-b)$. Furthermore, choose a point $r\in E\setminus \{b\}$ such that $\zeta:=\OO_E(b-r)$ is torsion of order precisely $\ell$. The scroll $X \to E$ has two sections $J_0$ respectively $J_1$, corresponding to the quotients $\mathcal{O}_E\oplus \eta \twoheadrightarrow \eta$ and $\mathcal{O}_E\oplus \eta \twoheadrightarrow \mathcal{O}_E$ respectively. We have $$J_1 \cong J_0-\phi^*\eta, \  \; \;  \; N_{J_0/X} \cong \OO_{J_0}(\phi^*\eta), \  \; \; \; N_{J_1/X} \cong \OO_{J_1}(\phi^*\eta^{\vee}),$$ where we freely mix notation for divisors and line bundles. For any point $x\in E$ we denote by $f_x$ the fibre $\phi^{-1}(x)$. We let $$C \in |gJ_0+f_r|$$ be a general element; this is a smooth curve of genus $g$. We further set $$L:=\OO_X\bigl((g-2)J_0+f_a\bigr).$$ Using that $K_X=-J_0-J_1$, the adjunction formula shows that the restriction $L_{C}$ is a level $\ell$  paracanonical bundle on $C$, that is, $[C, \tau]\in \cR_{g,\ell}$, where  $\tau:=\phi_C^*(\zeta) \cong L_C\otimes K_C^{\vee}$, with $\phi_C:C\rightarrow E$ being the restriction of $\phi$ to the curve $C$. In this paper we verify the Prym--Green Conjecture for this particular paracanonical curve of genus $g=2n+1$.

\vskip 4pt

Denoting by $\widetilde{X}$ the blow-up of $X$ at the two base points of $|L|$ and by $\widetilde{L}\in \mbox{Pic}(\widetilde{X})$ the proper transform of $L$, one begins by showing that the first vanishing
$K_{n-1,1}(C,K_C\otimes \tau)=0$ required in the Prym--Green Conjecture is a consequence of the vanishing of $K_{n-1,1}(\widetilde{X},\widetilde{L})$ and that of the mixed Koszul cohomology group $K_{n-2,2}(\widetilde{X}, -C,\widetilde{L})$ respectively (see Section \ref{defin} for details). By the Lefschetz hyperplane principle in Koszul cohomology, the vanishing of $K_{n-1,1}(\widetilde{X}, \widetilde{L})$ is a consequence of Green's Conjecture for a general curve $D$ in the linear system $|L|$ on $X$. Since $D$ has been proven in \cite{farkas-tarasca}  to be Brill-Noether general, Green's Conjecture holds for $D$. We then show (see (\ref{vanishing-kos})) that a sufficient condition for the second vanishing appearing in (\ref{pgeq}) is that
$$K_{n-2,2}\bigl(D,\OO_D(-C),K_D\bigr)=0 \ \ \mbox{ and } \ \ K_{n-1,2}\bigl(D, \OO_D(-C),K_D\bigr)=0.$$
Via results from \cite{farkas-mustata-popa} coupled with the usual description of Koszul cohomology in terms of kernel bundles, we prove that these vanishings are both consequences of the following transversality statement between difference varieties in the Jacobian $\mbox{Pic}^2(D)$
\begin{equation}\label{condd}
\OO_D(C)-K_D-D_2\nsubseteq D_n-D_{n-2},
\end{equation}
where, as usual, $D_m$ denotes the $m$-th symmetric product of $D$ (see Lemma \ref{suffcond1}). This last statement is proved inductively, using the canonical degeneration of $D$ inside its linear system
to a curve of lower genus with elliptic tails. It is precisely this feature of the elliptic surface $X$, of containing Brill-Noether general curves of \emph{every} genus
(something which is not shared by a $K3$ surface), which makes the proof possible. To sum up this part of the proof, we point that by using the geometry of $X$, we reduce the first half of the Prym--Green Conjecture, that is, the statement $K_{n-1,1}(C,K_C\otimes \tau)=0$ on the curve $C$ of genus $g$, to the geometric condition (\ref{condd}) on the curve $D$ of genus $g-2$.

\vskip 4pt

The second vanishing required by  the Prym--Green Conjecture, that is, $K_{n-3,2}(C,K_C\otimes \tau)=0$ falls in the range covered by the Green-Lazarsfeld \emph{Secant Conjecture} \cite{GL}. This feature appears only in odd genus, for even genus the Prym--Green Conjecture is \emph{beyond} the range in which the Secant Conjecture applies (see Section \ref{secant1} for details). For a curve $C$ of genus $g=2n+1$ and maximal Clifford index $\mbox{Cliff}(C)=n$,  the Secant Conjecture predicts that for a non-special line bundle $L\in \mbox{Pic}^{2g-2}(C)$, one has the following equivalence
$$K_{n-3,2}(C,L)=0\Longleftrightarrow L-K_C\notin C_{n-1}-C_{n-1}.$$
Despite significant progress, the Secant Conjecture is not known for arbitrary $L$, but in \cite{generic-secant} Theorem 1.7, we provided a sufficient condition for the vanishing to hold. Precisely, whenever
\begin{equation}\label{difftransl}
\tau+C_2\nsubseteq C_{n+1}-C_{n-1},
\end{equation} we have $K_{n-3,2}(C,K_C\otimes \tau)=0$. Thus the second half of the Prym--Green Conjecture has been reduced to a transversality statement of difference varieties
very similar to (\ref{condd}), but this time on the same curve $C$. Using the already mentioned elliptic tail degeneration inside the linear system $|C|$ on $X$, we establish inductively in Section \ref{secant1} that (\ref{difftransl}) holds for a general curve $C\subseteq X$ in its linear system. This completes the proof of the Prym--Green Conjecture.

\vskip 3pt

\noindent {\bf Acknowledgments:} The first author is supported by DFG Priority Program 1489 \emph{Algorithmische Methoden in Algebra, Geometrie und Zahlentheorie}. The second author is supported by NSF grant DMS-1701245 \emph{Syzygies, Moduli Spaces, and Brill-Noether Theory}.

\section{Elliptic surfaces and paracanonical curves}\label{defin}
We fix a level $\ell\geq 2$ and  recall that pairs $[C,\tau]$, where $C$ is a smooth curve of genus $g$ and $\tau\in \mbox{Pic}^0(C)$ is an $\ell$-torsion point, form an irreducible moduli space $\cR_{g,\ell}$. We refer to \cite{CEFS} for a detailed description of the Deligne-Mumford compactification $\rr_{g,\ell}$ of $\cR_{g,\ell}$.

\vskip 3pt

Normally we prefer multiplicative notation for line bundles, but occasionally, in order to simplify calculations, we switch to additive notation and identify divisors and line bundles. If $V$ is a vector space and $S:=\mbox{Sym } V$, for a graded $S$-module $M$ of finite type, we denote by $K_{p,q}(M,V)$ the Koszul cohomology group of $p$-th syzygies of weight $q$ of $M$. If $X$ is a projective variety, $L$ is a line bundle and $\mathcal{F}$ is a sheaf on $X$, we set as usual $K_{p,q}(X,\mathcal{F},L):=K_{p,q}\bigl(\Gamma_X(\mathcal{F},L), H^0(X,L)\bigr)$, where $\Gamma_X(\mathcal{F},L):=\bigoplus_{q\in \mathbb Z} H^0\bigl(X,\mathcal{F}\otimes L^{\otimes q}\bigr)$ is viewed as a graded $\mbox{Sym } H^0(X,L)$-module. For background questions on Koszul cohomology, we refer to the book \cite{aprodu-nagel}.

\vskip 3pt

Assume now that $g:=2n+1$ is odd and let us consider the decomposable elliptic ruled surface $\phi:X\rightarrow E$ defined in the Introduction. Retaining all the notation, our first aim is to establish the vanishing of the linear syzygy group $ K_{n-1,1}(C, K_C\otimes \tau)$. Before proceeding, we confirm that $\tau:=\phi_C^*(\zeta)$ is non-trivial of order precisely $\ell$, so that $[C,\tau]$ is indeed a point of $\cR_{g,\ell}$.

\begin{lem}
For any $1 \leq m \leq \ell-1$, the line bundle $\tau^{\otimes m} \in \mathrm{Pic}^0(C)$ is not effective.
\end{lem}
\begin{proof}
Since the order of $\zeta$ is precisely $\ell$, we  have $H^0(X,\phi^*(\zeta^{\otimes m}))\cong H^0(E,\zeta^{\otimes m})=0$ for $1 \leq m \leq \ell-1$. So it suffices to show $H^1\bigl(X,\phi^*(\zeta^{\otimes m})(-C)\bigr)=0$. By Serre duality, this is equivalent to $H^1\bigl(X,\phi^*(r+\eta-m\zeta)((g-2)J_0)\bigr)=0$. Applying the Leray spectral sequence this amounts to $$H^1\Bigl(E,\OO_E\bigl(a+(m+1)r-(m+1)b\bigr)\otimes \text{Sym}^{g-2}(\mathcal{O}_E\oplus\eta)\Bigr)=0,$$ which is clear for degree reasons.
\end{proof}

Any linear system which is a sum of a positive multiple of $J_0$ and a fibre of $\phi$ has two base points, see \cite{farkas-tarasca}, Lemma 2. In particular, the linear system $|L|$ on $X$ has two base points $p \in J_1$ and $q^{(g-2)} \in J_0$. Here
$$\{p\}:=f_a\cdot J_1 \ \mbox{ and } \ \{q^{(g-2)}\}:=f_{s^{(g-2)}}\cdot J_0,$$
where the point $s^{(g-2)}\in E$ is determined by the condition $\eta^{\otimes (g-2)}\cong \OO_E\bigl(s^{(g-2)}-a\bigr)$.

\vskip 3pt

Let $\pi:\widetilde{X}\to X$ be the blow-up of $X$ at these two base points, with exceptional divisors $E_1$ respectively \ $E_2$ over $p$ respectively \ $q^{(g-2)}$. We denote by  $\widetilde{L}:=\pi^*L-E_1-E_2$ the proper transform of $L$.
Note that $K_{\widetilde{X}}=-\widetilde{J}_0-\widetilde{J}_1$, where $\widetilde{J}_0=J_0-E_2$ and $\widetilde{J}_1=J_1-E_1$ are the proper transforms of $J_0$ and  $J_1$. We now observe that the base points of the two linear systems $|L|$ and $|C|$ on $X$ are disjoint.

\begin{lem}
Let $x_0 \in J_0$ and  $x_1 \in J_1$ be the two base points of $|C|$. Then $x_0,x_1 \notin \{ p,q^{(g-2)}\}$.
\end{lem}
\begin{proof}
First, since $r\neq a$, we obtain that $J_1 \cap f_a \neq J_1 \cap f_r$, therefore $p \neq x_1$. Next, recall that $\{q^{(g-2)}\}=J_0 \cap f_{s^{(g-2)}}$, where $\OO_E(s^{(g-2)}-a)=\eta^{\otimes (g-2)}$ and $\{x_0\}= J_0 \cap f_{t^{(g)}}$, where the point $t^{(g)}\in E$ is determined by the equation $\OO_E(t^{(g)}-r)=\eta^{\otimes g}$. We need to show $\eta^{\otimes (g-2)}(a) \neq \eta^{\otimes g}(r)$. Else, since  $\OO_E(a-r)=\eta\otimes \zeta$, it would imply $\zeta=\eta$, which is impossible, for $\zeta$ is a torsion class, whereas $\eta$ is not.
\end{proof}

Since the curve $C$ does not pass through the points $p$ and $q^{(g-2)}$ which are blown-up, we shall abuse notation by writing $C$ for $\pi^*(C)$. We set $S:=\text{Sym}\ H^0(\widetilde{X},\widetilde{L})$ and consider the short exact sequence of graded $S$-modules
$$0 \longrightarrow  \bigoplus_{q \in \mathbb{Z}} H^0(\widetilde{X},q\widetilde{L}-C) \longrightarrow \bigoplus_{q \in \mathbb{Z}} H^0(\widetilde{X},q\widetilde{L}) \longrightarrow M \longrightarrow 0, $$
where the first map is defined by multiplication with the section defining $C$ and the module $M$ is defined by this exact sequence. By the corresponding long exact sequence in Koszul cohomology, see \cite{green-koszul} Corollary 1.d.4, that is,
$$\cdots \longrightarrow  K_{p,1}(\widetilde{X}, \widetilde{L})\longrightarrow K_{p,1}\bigl(M,H^0(\widetilde{X}, \widetilde{L})\bigr)\longrightarrow K_{p-1,2}(\widetilde{X},-C,\widetilde{L})\longrightarrow \cdots,$$ the vanishing of the Koszul cohomology group $K_{p,1}\bigl(M,H^0(\widetilde{X}, \widetilde{L})\bigr)$ follows from $K_{p,1}(\widetilde{X},\widetilde{L})=0$ and $K_{p-1,2}(\widetilde{X},-C,\widetilde{L})=0$. The reason we are interested in the Koszul cohomology of $M$ becomes apparent in the following lemma:

\begin{lem}
We have the equality $K_{p,1}\bigl(M,H^0(\widetilde{X}, \widetilde{L})\bigr)\cong K_{p,1}(C,K_C\otimes \tau)$, for every $p \geq 0$.
\end{lem}
\begin{proof}
The restriction map induces an isomorphism $H^0(\widetilde{X},\widetilde{L}) \cong H^0(C,K_C\otimes \tau)$. First of all, note that the restriction map is injective, since $\widetilde{L}-C=\pi^*(-2J_0+f_a-f_r)-E_1-E_2$ is not effective (as it has negative intersection with the nef class $\pi^*(f_r)$). Next, $h^0(\widetilde{X},\widetilde{L})=h^0(X,L)=g-1$ by a direct computation using the projection formula, see also \cite{farkas-tarasca}, Lemma 2. As $h^0(C,K_C\otimes \tau)=g-1$, the restriction to $C\subseteq \widetilde{X}$ induces the claimed isomorphism.

Let $M_q$ denote the $q$-th graded piece of $M$. We have an isomorphism $M_0\cong H^0(\widetilde{X},\mathcal{O}_{\widetilde{X}})$ and we have already seen that $H^0(\widetilde{X},\widetilde{L}-C)=0$, so $M_1\cong H^0(\widetilde{X},\widetilde{L})\cong H^0(C,K_C\otimes \tau)$.  So we have the following commutative diagram
$$\small{\xymatrix{
\bigwedge^{p+1} H^0(\widetilde{L}) \otimes M_0 \ar[r]^{} \ar[d] &\bigwedge^{p} H^0(\widetilde{L}) \otimes M_1
\ar[r]^{\delta_1 \; \; \;} \ar[d] &\bigwedge^{p-1} H^0(\widetilde{L}) \otimes M_2  \ar[d] \\
\bigwedge^{p+1} H^0(K_C+ \tau)    \ar[r]^{} &\bigwedge^{p} H^0(K_C+\tau) \otimes H^0(K_C+ \tau)\ar[r]^{\delta'_1 \; \; \;} & \bigwedge^{p-1} H^0(K_C+ \tau) \otimes H^0(2K_C+ 2\tau)}
}$$
where the two leftmost vertical maps are isomorphisms and the rightmost vertical map is injective. Thus the middle cohomology of each row is isomorphic, so that we have the equality $K_{p,1}(M,H^0(\widetilde{X}, \widetilde{L}))\cong K_{p,1}(C,K_C\otimes \tau)$, for any $p \geq 0$.
\end{proof}

\vskip 4pt

\subsection{The vanishing of the Koszul cohomology group $K_{n-1,1}(C,K_C\otimes \tau)$.}

We can summarize the discussion so far. In order to establish the first vanishing required by the Prym--Green Conjecture for the pair $[C, \tau]$, that is, $K_{n-1,1}(C, K_C\otimes \tau)=0$, it suffices to prove that
\begin{align}
K_{n-1,1}(\widetilde{X},\widetilde{L})&=0, \: \; \text{and} \\
K_{n-2,2}(\widetilde{X},-C,\widetilde{L})&=0.
\end{align}
The first vanishing is a consequence of Green's Conjecture on syzygies of canonical curves.
\begin{prop}
We have $K_{n-1,1}(\widetilde{X},\widetilde{L})=0$.
\end{prop}
\begin{proof}
Let $D \in |\widetilde{L}|$ be a general element, thus $D$ is a smooth curve of genus $2n-1$. We have an isomorphism $K_{n-1,1}(\widetilde{X},\widetilde{L})\cong K_{n-1,1}(D,K_D)$, as $K_{\widetilde{X}|_{D}} \cong \mathcal{O}_D$ and by applying \cite{aprodu-nagel}, Theorem 2.20 (note that one only needs that the restriction $H^0(\widetilde{X},\widetilde{L})\to H^0(D,K_D)$ is surjective, and not $H^1(\widetilde{X}, \mathcal{O}_{\widetilde{X}})=0$, for this result). As $D$ is a smooth curve of genus $2n-1$, the vanishing in question is a consequence of Green's Conjecture, which is known to hold for curves of maximal gonality, see \cite{V2}, \cite{hirsch}. Hence it suffices to show that $D$ has maximum gonality $n+1$. But $D$ is the strict transform of a smooth curve in $|L|$ and is a Brill--Noether general curve by \cite{farkas-tarasca} Remark 2, in particular it has maximal gonality.
\end{proof}

We now turn our attention to the vanishing of the second Koszul group $K_{n-2,2}(\widetilde{X},-C, \widetilde{L})$. The following argument is inspired by \cite{green-koszul}, Theorem 3.b.7.
\begin{prop}
Let $D \in |\widetilde{L}|$ be general and let $p \geq 0$. Assume $K_{m,2}(D,\mathcal{O}_{D}(-C),K_{D})=0$ for $m \in \{p,p+1\}$. Then $$K_{p,2}(\widetilde{X},-C,\widetilde{L}) =0.$$
\end{prop}
\begin{proof}
Set as before  $S:=\text{Sym}\ H^0(\widetilde{X},\widetilde{L})$ and consider the exact sequence of graded $S$-modules
$$ 0 \longrightarrow \bigoplus_{q \in \mathbb{Z}} H^0(\widetilde{X}, (q-1)\widetilde{L}-C) \longrightarrow \bigoplus_{q \in \mathbb{Z}} H^0(\widetilde{X}, q\widetilde{L}-C)\longrightarrow B \longrightarrow 0, $$
serving as a definition for $B$, and where the first map is given by multiplication by a general section $s \in H^0(\widetilde{X},\widetilde{L})$. We now argue along the lines
of  \cite{generic-secant} Lemma 2.2. Taking the long exact sequence in Koszul cohomology and using that multiplication by a section $s \in H^0(\widetilde{X},\widetilde{L})$ induces the zero map on Koszul cohomology, we get
$$ K_{p,q}\bigl(B,H^0(\widetilde{X},\widetilde{L})\bigr) \cong K_{p,q}(\widetilde{X},-C,\widetilde{L}) \oplus K_{p-1,q}(\widetilde{X},-C,\widetilde{L}),$$
for all $p,q \in \mathbb{Z}$.

Let $D=Z(s)$ be the divisor defined by $s \in H^0(\widetilde{X},\widetilde{L})$, and consider the graded $S$-module
$$ N:= \bigoplus_{q \in \mathbb{Z}} H^0(D, qK_D-C_D). $$ We have the inclusion $B \subseteq N$ of graded $S$ modules. We claim $B_1=N_1=0$. By intersecting with the nef class $f_r$, we see $H^0(\widetilde{X}, \widetilde{L}-C)=0$, implying $B_1=0$. As $\deg(K_D-C_D)=-4$, we have $N_1=0$. Upon taking Koszul cohomology, this immediately gives the inclusion

$$K_{p,2}\bigl(B,H^0(\widetilde{X},\widetilde{L})\bigr) \subseteq K_{p,2}\bigl(N,H^0(\widetilde{X},\widetilde{L})\bigr).$$
In particular, $K_{p,2}(\widetilde{X},-C,\widetilde{L}) \subseteq K_{p+1,2}\bigl(B,H^0(\widetilde{X},\widetilde{L})\bigr)\subseteq K_{p+1,2}\bigl(N,H^0(\widetilde{X},\widetilde{L})\bigr)$.

\vskip 4pt

To finish the proof, it will suffice to show
\begin{equation}\label{splitting}
K_{p,2}\bigl(N,H^0(\widetilde{X},\widetilde{L})\bigr) \cong K_{p,2}\bigl(D,\mathcal{O}_{D}(-C),K_{D}\bigr) \oplus K_{p-1,2}\bigl(D,\mathcal{O}_{D}(-C),K_{D}\bigr).
\end{equation}
Since $\widetilde{L}\cdot \widetilde{J}_0=0$ and $\widetilde{L}\cdot \widetilde{J}_1=0$, it follows that $\OO_D(K_{\widetilde{X}})\cong \OO_D$. We now closely follow the proof of Lemma 2.2 in \cite{generic-secant}. The section $s$ induces a splitting $H^0(\widetilde{X}, \widetilde{L})\cong \mathbb C\{s\}\oplus H^0(D,K_D)$, giving rise for every $p$ to isomorphisms
$$\bigwedge^p H^0(\widetilde{X},\widetilde{L})\cong \bigwedge^{p-1} H^0(D,K_D)\oplus \bigwedge^p H^0(D,K_D).$$ The desired isomorphism (\ref{splitting}) follows from a calculation which is identical to the one carried out in the second part of the proof of \cite{generic-secant} Lemma 2.2. There one works with a $K3$ surface, but the only thing needed for the argument to work is that $\OO_D(K_{\widetilde{X}})\cong \OO_D$.
\end{proof}

\vskip 3pt

To establish that $K_{n-1,1}(C,K_C\otimes \tau)=0$, it thus suffices to show
\begin{align} \label{vanishing-kos}
K_{n-2,2}\bigl(D,\mathcal{O}_{D}(-C), K_{D}\bigr)=0 \; \; \text{and \;} K_{n-1,2}\bigl(D,\mathcal{O}_{D}(-C), K_{D}\bigr)=0.
\end{align}
Via a well-known description of Koszul cohomology using kernel bundles, cf. \cite{aprodu-nagel} Proposition 2.5, taking into account that $H^0(D,K_D-C_D)=0$, these two statements are equivalent to
\begin{align}\label{vanishing-kos2}
 H^0\Bigl(D,\bigwedge^{n-2}M_{K_D}\otimes(2K_D-C_D)\Bigr)=0  \; \; \text{and \;} H^0\Bigl(D,\bigwedge^{n-1}M_{K_D}\otimes (2K_D-C_D)\Bigr)=0,
 \end{align}
where we recall that  $M_{K_D}$ is the \emph{kernel bundle}, defined by the short exact sequence
$$ 0 \longrightarrow M_{K_D} \longrightarrow H^0(D,K_D) \otimes \mathcal{O}_D \longrightarrow K_D \longrightarrow 0.$$
Both statements (\ref{vanishing-kos2}) will be reduced to general position statements with respect to divisorial difference varieties of the various curves on $X$.
\vskip 3pt

\subsection{Containment between difference varieties on curves.} If $C$ is a smooth curve of genus $g$, we denote by $C_a-C_b\subseteq \mbox{Pic}^{a-b}(C)$ the image of the difference map $v:C_a\times C_b\rightarrow \mbox{Pic}^{a-b}(C)$. We  occasionally make use of the realization given in \cite{farkas-mustata-popa} of the \emph{divisorial} difference varieties as non-abelian theta divisors associated to exterior powers of the kernel bundle of $K_C$. Precisely, for $i=0, \ldots, \lfloor \frac{g-1}{2}\rfloor$, one has the following equality of divisors on $\mbox{Pic}^{g-2i-1}(C)$:
\begin{equation}\label{fmp2}
C_{g-i-1}-C_i=\Bigl\{\xi\in \mbox{Pic}^{g-2i-1}(C): H^0\Bigl(C,\bigwedge^i M_{K_C}\otimes K_C\otimes \xi^{\vee}\Bigr)\neq 0\Bigr\}.
\end{equation}

We now make an observation concerning a containment relation between difference varieties.

\begin{lem} \label{routine-difference}
Let $C$ be a smooth curve, $a\geq 2$, $b \geq 0$, $c >0$ be integers and $A\in \mathrm{Pic}^{a+b-c}(C)$.  Assume $A-C_a \subseteq C_b-C_c$. Then
$A-C_{a-2} \subseteq C_{b+1}-C_{c-1}$.
\end{lem}
\begin{proof}
Let $B$ be an arbitrary effective divisor of degree $a-2$, and let $y_0 \in C$ be a fixed point. Since $A-C_a \subseteq C_b-C_c$, we have a well-defined  morphism
\begin{align*}
f: C & \to C_b-C_c \subseteq \text{Pic}^{b-c}(C) \\
x &\mapsto A-(B+x+y).
\end{align*}
We further have the difference map $v:C_b \times C_c\rightarrow \mbox{Pic}^{b-c}(C)$ given by
$v(F_1,F_2):=\OO_C(F_1-F_2)$, where $F_1$ and $F_2$ are effective divisors of degrees $b$ and $c$ respectively, as well as the projection $p_2: C_b \times C_c \to C_c$.

\vskip 3pt

Suppose firstly that $\dim p_2\bigl(v^{-1}(\mbox{Im}(f))\bigr) \geq 1$. As the divisor $y_0+C_{c-1} \subseteq C_c$ is ample, see \cite{fulton-laz-connectedness} Lemma 2.7, $p_2\bigl(v^{-1}(\mbox{Im}(f))\bigr)$ must meet $y_0+C_{c-1}$.
This means that there exists a point $x \in C$ such that $A-(B+x+y)\equiv F_1-F_2$, with $F_1 \in C_b$ and $F_2 \in C_c$ being effective divisors such that $F_2=y_0+F_2'$, where $F_2'\in C_{c-1}$ is effective. But then
$$ A-B=\OO_C(F_1+x-F_2') \in C_{b+1}-C_{c-1}.$$

Assume now $p_2\bigl(v^{-1}(\mbox{Im}(f))\bigr) \subseteq C_c$ is finite. Then one can find a divisor $F_2 \in C_c$, such that for every $x \in C$, there is a divisor $F_x\in C_b$ with $A-B-x-y=F_x-F_2$.
Picking $x\in \mbox{supp}(F_2)$, we  write $F_2=x+F_2'$, where $F_2'\in C_{c-1}$. Then $A-B=\OO_C(F_x+y_0-F_2') \in C_{b+1}-C_{c-1}.$
\end{proof}

We may now restate the vanishing conditions (\ref{vanishing-kos}) in terms of difference varieties. From now on we revert to the elliptic surface $\phi:X\rightarrow E$ and recall that $C\in |gJ_0+f_r|$.

\begin{lem}\label{suffcond1}
Set $g=2n+1$ with $n\geq 2$ and choose a general curve $D \in |(g-2)J_0+f_a|$. Suppose
$$ C_D-K_D-D_2 \nsubseteq D_n-D_{n-2}.$$
Then $K_{n-1,1}(C,K_C\otimes \tau)=0$, for a general level $\ell$ curve $[C,\eta]\in \cR_{g,\ell}$.
\end{lem}
\begin{proof}
By assumption, there exist points $x,y\in D$ such that $C_D-K_D-x-y \notin D_n-D_{n-2}.$ It follows from  (\ref{fmp2}) that this is equivalent to
$H^0\bigl(D,\bigwedge^{n-2}M_{K_D}\otimes (2K_D-C_D+x+y)\bigr)=0$, implying $H^0\bigl(D,\bigwedge^{n-2}M_{K_D}\otimes (2K_D-C_D)\bigr)=0$. This is equivalent to $K_{n-2,2}(D,\mathcal{O}_{D}(-C),K_{D})=0$.

\vskip 3pt

Next, by Lemma \ref{routine-difference}, our assumption implies $C_D-K_D-D_4 \nsubseteq D_{n-1}-D_{n-1}$.  Thus $H^0\bigl(D,\bigwedge^{n-1}M_{K_D}\otimes (2K_D-C_D+T)\bigr)=0$, for some effective divisor $T\in D_4$, therefore $H^0\bigl(D,\bigwedge^{n-1}M_{K_D}\otimes (2K_D-C_D)\bigr)=0$ as well, amounting to $K_{n-1,2}\bigl(D,\mathcal{O}_{D}(-C), K_{D}\bigr)=0$.
\end{proof}

Any smooth divisor $D \in |L|$ carries two distinguished points, namely  $p$ and $q^{(g-2)}$. We will prove that, if $D \in |L|$ is general, then
\begin{equation}\label{pgtoprove}
C_D-K_D-p-q^{(g-2)} \notin D_n-D_{n-2}.
\end{equation}

Let us first introduce some notation. For an integer $m \geq 1$, we define the line bundle $$L_m:=\OO_X(mJ_0+f_a) \in \text{Pic}(X).$$ A general element $D \in |L_m|$ is a smooth curve of genus $m$, having two distinguished points $p \in J_1$ and $q^{(m)} \in J_0$, which as already explained, are the base points of $|L_m|$. Recall that for each $j=0,\ldots, m-1$, we introduced the divisorial difference variety
$$D_j-D_{m-1-j} \subseteq \text{Pic}^{2j+1-m}(D).$$
This difference variety is empty for $j<0$ or $j>m-1$.  We shall prove (\ref{pgtoprove})  inductively by contradiction, using the fact that in any family of curves on the surface $X$, there is a canonical degeneration to a curve with an elliptic tail.

\vskip 4pt

\subsection{The induction step} Assume that for a general curve $D\in |L_{g-2-j}|$ one has
$$C_D-K_D-p-(2i+1)q^{(g-2-j)} \in D_{n-i}-D_{n-2-j+i}, \text{\; for some $0 \leq i \leq j$}.$$
Then for a general curve $Z\in |L_{g-3-j}|$, one has

\begin{equation}\label{indtodo}
C_Z-K_Z-p-(2i'+1)q^{(g-3-j)} \in Z_{n-i'}-Z_{n-3-j+i'},  \text{\; for some $0 \leq i' \leq j+1$}.
\end{equation}
Notice that the assumption $D_{n-i}-D_{n-2-j+i}  \neq \emptyset$ for a curve $D\in |L_{g-2-j}|$ implies
\begin{align} \label{relevant-bds}
0 \leq n-i \leq g-3-j.
\end{align}

Let $D \in |L_{g-2-j}|$ be general. In order to prove the induction step, we degenerate $D$ within its linear system to the curve of compact type
$$Y:=J_0+Z,$$
for a general $Z \in |L_{g-3-j}|$. Notice that $J_0 \cdot Z=q^{(g-3-j)}=:q$ and the marked point $p$ lies on $Z\setminus \{q\}$. On $Y$, in the spirit of limit linear series, we choose the twist of bidegree $\bigl(0, 2g-2j-6\bigr)$ of its dualizing sheaf, that is, the line bundle
$$\widetilde{K} \in \text{Pic}(Y)$$
characterized  by $\widetilde{K}\otimes \OO_{J_0} \cong \mathcal{O}_{J_0}$  and $\widetilde{K}\otimes \OO_{Z} \cong K_{Z}(2q)$.
We establish a few technical statements to be used later in the proofs.

\begin{lem} \label{main-coh-lem}
Assume the bounds (\ref{relevant-bds}). Then, for any $0 \leq i \leq j \leq g-4$, we have:
\begin{enumerate} [label=(\roman*)]
\item \label{coh1-part1} $h^0(Y,\widetilde{K})=h^0(D, K_{D})=g-2-j.$
\item $H^0\Bigl(Y,\OO_Y(C-J_1-(2i+1)J_0)\otimes \widetilde{K}^{\vee}\Bigr)=0$
\item $h^0\Bigl(Y,\OO_Y\bigl(C-J_1-(2i+1)J_0\bigr)\Bigr)=h^0\Bigl(D,C_D\bigl(-p-(2i+1)q^{(g-2-j)}\bigr)\Bigr) $
\item $h^0\Bigl(Y,\OO_Y(C-J_1-(2i+1)J_0)\otimes \widetilde{K}\Bigr)=h^0\Bigl(D,C_D\otimes K_D\bigl(-p-(2i+1)q^{(g-2-j)}\bigr)\Bigr)$.
\end{enumerate}
\end{lem}

\begin{proof}
\mbox{}

\noindent (i) As $\widetilde{K}$ is a limit of canonical bundles on smooth curves,  $h^0(Y,\widetilde{K}) \geq g-2-j=h^0(D, K_{D})$. So it suffices to show $h^0(Y,\widetilde{K})\leq h^0(D, K_D)$. Twisting by $\widetilde{K}$ the short exact sequence
\begin{align} \label{MV} 0 \longrightarrow \mathcal{O}_{J_0}(-q) \longrightarrow \mathcal{O}_{Y} \longrightarrow \mathcal{O}_{Z} \longrightarrow 0
\end{align}  and taking cohomology, we get $h^0(Y,\widetilde{K})  \leq h^0(Z,K_{Z}(2q))=g-2-j$, as required.

\vskip 4pt

\noindent (ii) Set $A_d:=\OO_Y\bigl(C-J_1-(2i+1)J_0\bigr)\otimes \widetilde{K}^{\otimes d} \in \text{Pic}(Y)$. One needs to show $H^0(Y,A_{-1})=0$. Via the projection $\phi:X\rightarrow E$ we identify the section $J_0$ with the elliptic curve $E$. We have
$\OO_{J_0}(A_{-1}) \cong \eta^{\otimes (g-2i-1)}(r)$. Furthermore
$\OO_{J_0}(q) \cong \eta^{\otimes (g-3-j)}(a)$, hence
$$\OO_{J_0}(A_{-1}(-q)) \cong \eta^{\otimes (j-2i+2)}(r-a).$$ We have
$H^0(E, \eta^{\otimes (j-2i+2)}(r-a))=H^0\bigl(E,\zeta^{\vee}\otimes \eta^{\otimes (j-2i+1)}\bigr)=0$, for $\zeta$ is $\ell$-torsion, whereas $\eta$ is not a torsion bundle. From the short exact sequence (\ref{MV}) twisted by $A_{-1}$, in order to conclude it suffices to show that the restricted line bundle
\begin{align*}
\OO_{Z}({A_{-1}}) &\cong \OO_{Z}\bigl((g-2i-3)J_0-J_1+f_r\bigr)\otimes K_{Z}^{\vee}\\
& \cong \OO_{Z}\bigl((j+1-2i)J_0+f_r-f_a\bigr)
\end{align*}
is not effective. We will firstly show $H^0(X,(j+1-2i)J_0+f_r-f_a)=0$. If $j+1-2i <0$, this is immediate since then $\bigl((j+1-2i)J_0+f_r-f_a\bigr) \cdot f_r<0$ and the curve $f_r$ is nef. If $j+1-2i \geq 0$, we use the isomorphism
$$H^0(X,(j+1-2i)J_0+f_r-f_a) \cong H^0\bigl(E,\mathcal{O}_E(r-a) \otimes \text{Sym}^{j+1-2i}(\mathcal{O}_E\oplus \eta)\bigr)=0. $$ In order to conclude, it is enough to show
$H^1(X,(j+1-2i)J_0+f_r-f_a-Z)=0$. By Serre duality, this is equivalent to
$$H^1(X,K_X+Z+f_a-f_r-(j+1-2i)J_0)=0. $$ We compute
$$ K_X+Z+f_a-f_r-(j+1-2i)J_0=(g-6+2i-2j)J_0+\phi^*\eta+2f_a-f_r,$$
where $g-6+2i-2j \geq -1$ by (\ref{relevant-bds}).
If $g-6+2i-2j \geq 0$, then
$$ H^1(X,(g-6+2i-2j)J_0+\phi^*\eta+2f_a-f_r)= H^1\bigl(E,\mathcal{O}_E(2a-r+\eta) \otimes \text{Sym}^{{g-6+2i-2j}}(\mathcal{O}_E \oplus \eta)\bigr),$$
which vanishes for degree reasons. Finally, if $g-6+2i-2j=-1$, an  application of the Leray spectral sequence implies $H^1(X, -J_0+\phi^*\eta+2f_a-f_r)=0$, as well. This completes the proof.

\vskip 4pt

\noindent (iii) By Riemann--Roch and semicontinuity, it suffices to show $H^1(Y,A_0)=0$, that is,
$$H^1\bigl(Y, \OO_Y((g-2i-1)J_0-J_1+f_r)\bigr)=0. $$

If so, then the bundle $\OO_X(C-J_1-(2i+1)J_0)$ has the same number of sections, when restricted to a general element $D\in |L_{g-2-j}|$ or to its codimension $1$ degeneration $Y$ in its linear system.
By $(\ref{relevant-bds})$, we have $g-2i-1\geq 0$. First, starting from $H^1(X,-J_1+f_r)=0$, which is an easy consequence of the Leray spectral sequence, one shows inductively that $H^1(X, mJ_0-J_1+f_r)=0$ for all $m\geq 0$, in particular also
$H^1\bigl(X, (g-2i-1)J_0-J_1+f_r\bigr)=0$.

\vskip 3pt

To conclude, it is enough to show $H^2\bigl(X, (g-2i-1)J_0-J_1+f_r-Y\bigr)=0.$ By Serre duality,
$$H^2\bigl(X, (g-2i-1)J_0-J_1+f_r-Y\bigr)\cong H^0\bigl(X,(2i-2-j)J_0+f_a-f_r\bigr)^{\vee}.$$

If $2i-2-j<0$, then the class $(2i-2-j)J_0+f_a-f_r$ is not effective on $X$ as it has negative intersection with $f_r$. If $2i-2-j \geq 0$ then this class is not effective by projecting to $E$.

\vskip 4pt

\noindent (iv) It suffices to show $H^1(Y, A_1)=0$. We use the exact sequence on $Y$
$$ 0 \longrightarrow  \mathcal{O}_{Z}(-q) \longrightarrow \mathcal{O}_{Y} \longrightarrow \mathcal{O}_{J_0} \longrightarrow  0.$$
As $\deg \OO_{J_1}({A_1})=1$, it is enough to show $H^1\bigl(Z, \OO_{Z}(A_1)(-q)\bigr)=0$. By direct computation
$$ \deg \OO_{Z}(A_1(-q))=\deg K_{Z}+2n-2i+g-3-j.$$ From (\ref{relevant-bds}), $n-i \geq 0$, whereas $j \leq g-4$ by assumption, so $g-3-j>0$ and
the required vanishing follows for degree reasons.
\end{proof}

\vskip 4pt

We now have all the pieces needed to prove the induction step. The transversality statement (\ref{condd})  the first half of the Prym--Green Conjecture has been reduced to, is proved inductively, by being part of a system of condition involving difference varieties of curves of every genus on the surface $X$.

\begin{prop}\label{inductionstep}
Fix $0 \leq j \leq g-3$ and assume that for a general curve $D\in |L_{g-2-j}|$
$$C_D-K_D-p-(2i+1)q^{(g-2-j)} \in D_{n-i}-D_{n-2-i+j}, \ \mbox{ for some } 0 \leq i \leq j.$$
Then for a general curve $Z\in |L_{g-3-j}|$, the following holds
$$C_Z-K_Z-p-(2i'+1)q^{(g-3-j)} \in Z_{n-i'}-Z_{n-3-j+i'}, \ \mbox{ for some } 0 \leq i' \leq j+1. $$
\end{prop}
\begin{proof}
Using the determinantal realization of divisorial varieties (\ref{fmp2}) emerging from \cite{farkas-mustata-popa},  the assumption may be rewritten as
$$H^0\Bigl(D, \bigwedge^{n-2-j+i}M^{\vee}_{K_{D}} \otimes K_D^{\vee} \otimes \OO_D\bigl(C-p-(2i+1)q^{(g-2-j)}\bigr)\Bigr) \neq 0,$$
or, equivalently,
$$H^0\Bigl(D, \bigwedge^{n-i}M_{K_{D}} \otimes \OO_D(C-J_1-(2i+1)J_0)\Bigr) \neq 0.$$ By Lemma \ref{main-coh-lem} (ii), $H^0\bigl(D, \OO_D(C-J_1-(2i+1)J_0)\otimes K_{D}^{\vee}\bigr)=0$, so this is amounts to
$$K_{n-i,0}\bigl(D,\OO_D(C-J_1-(2i+1)J_0), K_{D}\bigr) \neq 0. $$
We now let $D$ degenerate inside its linear system to the curve $Y=J_0+Z$, where $Z\in |L_{g-3-j}|$ and $J_0\cdot Z=q^{(g-3-j)}=:q$.
By semicontinuity for Koszul cohomology \cite{BG}, together with Lemma \ref{main-coh-lem}, this implies
$$K_{n-i,0}\bigl(Y,\OO_Y(C-J_1-(2i+1)J_0), \widetilde{K}\bigr) \neq 0,$$
where $\widetilde{K}$ is the twist of the dualizing sheaf of $Y$ introduced just before Lemma \ref{main-coh-lem} . This is the same as saying that the map
$$\bigwedge^{n-i}H^0(Y, \widetilde{K}) \otimes H^0(Y,A_0) \to \bigwedge^{n-i-1}H^0(Y, \widetilde{K}) \otimes H^0(Y,A_1)  $$
is not injective, where recall to have defined the line bundles $A_d:=\OO_Y(C-J_1-(2i+1)J_0)\otimes \widetilde{K}^{\otimes d}$. As seen in the proof of Lemma \ref{main-coh-lem} (i), restriction induces an isomorphism
$$H^0(Y, \widetilde{K}) \cong H^0\bigl(Z, K_{Z}(2q)\bigr).$$
Using the identification between $J_0$ and $E$, we have seen in the proof of Lemma \ref{main-coh-lem} (ii) that  $\OO_{J_0}(A_d)(-q) \cong \eta^{\otimes (j-2i+2)}(r-a)$ is a nontrivial line bundle of degree $0$ on $E$, therefore
$H^i\bigl(J_0,\OO_{J_0}{A_d}(-q)\bigr)=0$ for $i=0,1$. Thus, restriction to $Z$ induces an isomorphism
$$H^0(Y,A_d) \cong H^0(Z,\OO_Z(A_d)).$$
This also gives that the map
$$
\bigwedge^{n-i}H^0\bigl(Z,K_{Z}(2q)\bigr) \otimes H^0\bigl(Z,\OO_Z(A_0)\bigr) \rightarrow \bigwedge^{n-i-1} H^0\bigl(Z, K_{Z}(2q)\bigr) \otimes H^0\bigl(Z,\OO_Z(A_1)\bigr), $$
fails to be injective. As one has
$$H^0(Z, \OO_Z(A_1)) \subseteq H^0\bigl(Z,\OO_Z(A_1+2q)\bigr) \ \mbox{ and } \  H^0\bigl(Z,\OO_Z(A_{-1}-2q)\bigr)=0,$$
we obtain $K_{n-i,0}\bigl(Z,\OO_Z(A_0), K_{Z}(2q)\bigr) \neq 0$,
which can be rewritten as
\begin{equation}\label{beauville}
H^0\Bigl(Z,\bigwedge^{n-i}M_{K_{Z}(2q)} \otimes \OO_Z(A_0)\Bigr) \neq 0.
\end{equation}

\vskip 3pt

We compute the slope $\mu\Bigl(\bigwedge^{n-i} M_{K_{Z}(2q)}\otimes \OO_Z(A_0)\Bigr)=g(Z)-1,$
where  $\mu\bigl(M_{K_Z(2q)}\bigr)=-2$. By Serre-Duality, then condition (\ref{beauville}) can be rewritten as
$$H^0\Bigl(Z,\bigwedge^{n-i}M^{\vee}_{K_Z(2q)} \otimes K_{Z}\otimes \OO_Z(-A_0)\Bigr) \neq 0.$$ We now use that Beauville in \cite{beauville-stable} Proposition 2 has described the theta divisors of vector bundles of the form $\bigwedge^{n-i} M_{K_Z(2q)}$ as above. Using \cite{beauville-stable}, from (\ref{beauville}) it follows that either
$$ \OO_Z(A_0)-  K_{Z} \in Z_{n-i}-Z_{n-3+i-j}, $$
or else
$$\OO_Z(A_0)- K_{Z}-2q \in Z_{n-i-1}-Z_{n-2+i-j}.$$
Taking into account that $\OO_Z(A_0)=\OO_Z(C-p-(2i+1)q)$, the desired conclusion now follows. As a final remark, we note that, whilst in \cite{beauville-stable} it is  assumed that $Z$ is non-hyperelliptic (which, using the Brill-Noether genericity of $Z$, happens whenever $g-3-j \geq 3$),  the above statement is a triviality for $g-3-j=1$, whereas in the remaining case $g-3-j=2$ it follows directly from the argument in \cite{beauville-stable} Proposition 2. Indeed, in this case we have a short exact sequence
$$0 \longrightarrow \bigwedge^{n-i-1}M^{\vee}_{K_Z}(2q) \longrightarrow \bigwedge^{n-i}M^{\vee}_{K_Z(2q)} \longrightarrow \bigwedge^{n-i}M^{\vee}_{K_Z}\longrightarrow 0. $$ The claim now follows immediately from \cite{farkas-mustata-popa}  \S 3. This completes the proof.
\end{proof}

\vskip 3pt

By the above Proposition and induction, we now reduce the proof of (\ref{vanishing-kos}) to a single statement on elliptic curves on the  ruled surface $X$:

\begin{thm}\label{pgpartone}
Set $g=2n+1$ and $\ell\geq 2$. Then for a general element $[C,\tau]\in \cR_{g,\ell}$, one has
$K_{n-1,1}(C,K_C\otimes \tau)=0$.
\end{thm}

\begin{proof}
We apply Lemma \ref{suffcond1} and the sufficient condition (\ref{pgtoprove}). By the inductive step described above, reasoning by contradiction, it suffices to show that if $D\in |L_1|$ is general, then
\begin{align*}
C_{D}-K_{D}-p-(2i+1)q^{(1)} \notin D_{n-i}-D_{n-i},
\end{align*}
for each $1 \leq i \leq g-3$. Assume this is not the case. The bounds (\ref{relevant-bds}) force $n=i$ and the difference variety on the right consists of $\{\OO_{D}\}$.  One needs to prove that
$\OO_{D}(C-p-(2n+1)q^{(1)}) \cong \OO_{D}({f_r}-J_1)$ is not effective. As $H^0(X,f_r-J_1)=0$ and $D\in |J_0+f_a|$, it suffices to prove
$$ H^1(X,f_r-J_1-J_0-f_a)=H^1(X,f_r-f_a+K_X)=0,$$
or equivalently by Serre duality, that $H^1(X,f_a-f_r)=0$. This follows immediately from the Leray spectral sequence.
\end{proof}

\section{The Green-Lazarsfeld Secant Conjecture for paracanonical curves}\label{secant1}
We recall the statement of the Green-Lazarsfeld Secant Conjecture \cite{GL}. Let $p$ be a positive integer, $C$ a smooth curve of genus $g$ and  $L$ a non-special line bundle of degree
\begin{equation}\label{bound}
d\geq 2g+p+1-\mbox{Cliff}(C).
\end{equation}
The Secant Conjecture predicts that if $L$ is $(p+1)$-very ample then $K_{p,2}(C,L)=0$ (the converse implication is easy, so one has an equivalence). The Secant Conjecture has been proved in many cases in \cite{generic-secant}, in particular for a general curve $C$ and a general line bundle $L$. In the extremal case $d=2g+p+1-\mbox{Cliff}(C)$, Theorem 1.7 in \cite{generic-secant} says that whenever
$$L-K_C+C_{d-g-2p-3}\nsubseteq C_{d-g-p-1}-C_{2g-d+p}$$ (the left hand side being a divisorial difference variety), then $K_{p,2}(C,L)=0$.
Theorem 1.7 in \cite{generic-secant} requires $C$ to be Brill-Noether-Petri general, but the proof given in \emph{loc. cit.} shows that for curves of odd genus the only requirement is that $C$ have maximum gonality $\frac{g+3}{2}$.

\vskip 4pt

In the case at hand, we choose a general curve on the decomposable elliptic surface $X$ $$C \in |gJ_0+f_r|$$ of genus $g=2n+1$ and Clifford index $\mbox{Cliff}(C)=n$. We apply the above result  to $L_{C}=K_C\otimes \tau$, with $\tau=\OO_C(\zeta)$. In order to conclude that $K_{n-3,2}(C,K_C\otimes \tau)=0$, it suffices to show
$$\tau+C_2 \nsubseteq  C_{n+1}-C_{n-1}.$$
 We have two natural points on $C$, namely those cut out by intersection with $J_0$ and $J_1$, and for those points it suffices to show
\begin{equation}\label{pgtoshow2}
\phi^*\zeta\otimes \OO_C(J_0+J_1) \notin C_{n+1}-C_{n-1}.
\end{equation}

\vskip 4pt

We first establish a technical result similar to Lemma \ref{main-coh-lem} and because of this analogy we use similar notations. For $0 \leq i \leq g-1$, let $Y \in |(g-i)J_0+f_r|$ be the union of $J_0$ and a general curve $Z \in |(g-i-1)J_0+f_r|$. We set $x_0:=Z\cdot J_0=f_{t^{g-i-1}}\cdot J_0$, where $t^{(g-i-1)}\in E$ satisfies $\OO_E(t^{(g-i-1)}-r)=\eta^{\otimes (g-i-1)}$. We denote $\widetilde{K} \in \text{Pic}(Y)$ the twist at the node of the dualizing sheaf of $Y$  such that  $\OO_{J_0}(\widetilde{K})\cong  \mathcal{O}_{J_0}$ and $\OO_{Z}(\widetilde{K})\cong K_Z(2x_0)$. We recall that $Z$ has a second distinguished point, namely $x_1=Z\cdot J_1$.

\vskip 4pt

\begin{lem} \label{coh-lem-quad}
Let $Y \in |(g-i)J_0+f_r|$ for $0 \leq i \leq g-1$ be as above and assume $j$ is an integer satisfying $0 \leq j \leq n-1$ and $0 \leq i-j \leq n+1$. For a general $D\in |(g-i)J_0+f_r|$, we have:
\begin{enumerate} [label=(\roman*)]\label{coh2-part1}
\item $h^0(Y,\widetilde{K})=h^0(D,K_{D})$.
\item $H^0\Bigl(Y,\OO_Y\bigl(\phi^*\zeta^{\vee}-(2j+1-i)J_0-J_1\bigr)\Bigr)=0$
\item $h^0\Bigl(Y,\widetilde{K}^{\otimes m}\otimes \phi^*{\zeta^{\vee}}(-(2j+1-i)J_0-J_1)\Bigr)=h^0\Bigl(D,K_{D}^{\otimes m}\otimes \phi^*{\zeta^{\vee}}(-(2j+1-i)J_0-J_1)\Bigr)$, for $m \in \{1,2\}$.
\end{enumerate}
\end{lem}
\begin{proof}
\mbox{}
(i) This is similar to Lemma \ref{main-coh-lem} (i) and we skip the details.

\vskip 3pt

\noindent (ii)  Set $k=-(2j+1-i)$. If $k \leq 0$, the statement is clear for degree reasons, so assume $k \geq 1$. Then
$H^0\bigl(X,\phi^*\zeta^{\vee}\otimes \OO_X(k J_0-J_1)\bigr)\cong H^0\bigl(E,(\eta\otimes \zeta^{\vee})\otimes\text{Sym}^{k-1}(\mathcal{O}_E \oplus \eta)\bigr)=0,$
so it suffices to show that $H^1\bigl(X,\phi^*\zeta^{\vee}\otimes \OO_X(kJ_0-J_1-Y)\bigr)=0$. By Riemann--Roch, this is equivalent to
$$H^1\bigl(X,\phi^*\zeta\otimes \OO_X((g-k-i-1)J_0+f_r)\bigr)=0.$$ Using the given bounds, $g-k-i-1\geq -1$. It suffices to show $H^1\bigl(X,\phi^*\zeta\otimes \OO_X(mJ_0+f_r)\bigr)=0$, for $m \geq -1$. This follows along the lines of the proof of Lemma \ref{main-coh-lem}.


\vskip 3pt

\noindent (iii) By Riemann--Roch and semicontinuity, it suffices to show that for $m=1,2$, one has
$$H^1\Bigr(Y,\widetilde{K}^{\otimes m}\otimes \OO_Y(\phi^*{\zeta^{\vee}}-(2j+1-i)J_0-J_1)\Bigr)=0.$$ Consider the exact sequence

\begin{equation}\label{exseq}
0 \longrightarrow \mathcal{O}_{Z}(-x_0) \longrightarrow  \mathcal{O}_{Y} \longrightarrow  \mathcal{O}_{J_0} \longrightarrow 0.
\end{equation}

Then $\OO_{J_0}\bigl(\widetilde{K}^{\otimes m}\otimes \phi^*{\zeta^{\vee}}(-(2j+1-i)J_0-J_1)\bigr)\cong  \zeta^{\vee}\otimes \eta^{\otimes (i-2j-1)}  \neq 0 \in \text{Pic}^0(E)$.
So it suffices to show $H^1\bigl(Z,K_{Z} \otimes \phi^*\zeta^{\vee}(-(2j-i)x_0-x_1)\bigr)=0$ and $H^1\bigl(Z,K_{Z}^{\otimes 2}\otimes \phi^*\zeta^{\vee}(-(2j-i-2)x_0-x_1)\bigr)=0$. The second  vanishing is automatic for degree reasons (using the bounds on $i$ and $j$), so we just need to establish the first one. By Serre duality, this is equivalent to
$$H^0\Bigl(Z, \phi^*(\zeta\otimes \eta^{\vee})\otimes \OO_Z((2j-i+1)x_0)\Bigr)=0.$$
This is obvious if $2j-i+1<0$, so assume $2j-i+1 \geq 0$. Using once more the Leray spectral sequence, it follows   $H^0\bigl(X,\phi^*(\zeta\otimes \eta^{\vee})\otimes \OO_X((2j-i+1) J_0)\bigr)=0$, so it suffices to prove
$$H^1\Bigl(X,\phi^*(\zeta\otimes \eta^{\vee})\otimes \OO_X\bigl((2j-i+1)J_0-Z\bigr)\Bigr)=0.$$  Using Serre duality and  the  bound $g-2j-4 \geq -1$,  this goes through  as in the proof of \ref{coh2-part1}.

\end{proof}

The following proposition provides the induction step, to be proved in order to establish the second half of the Prym--Green Conjecture:

\begin{prop}
Let $0 \leq i \leq g-2$. Suppose that for a general curve $D\in |(g-i)J_0+f_r|$ there exists an integer $0\leq j\leq i$ such that
$$\OO_{D}\bigl(\phi^*\zeta+(2j-i+1)J_0+J_1\bigr) \in D_{n+1-i+j}-D_{n-1-j}.$$
Then for a general curve $Z\in |(g-i-1)J_0+f_r|$, there exists $0\leq j'\leq i+1$ such that
$$\OO_Z\bigl(\phi^*\zeta+(2j'-i)J_0+J_1\bigr)  \in Z_{n-i+j'}-Z_{n-1-j'}.$$
\end{prop}
\begin{proof}
By assumption, $j \leq n-1$ and $i-j \leq n+1$. Applying again the determinantal description of divisorial difference varieties from \cite{farkas-mustata-popa} \S 3 and with Serre duality, the hypothesis turns into
$$H^0\Bigl(D, \bigwedge^{n-1-j}M_{K_D}\otimes \OO_D\bigl(\phi^*\zeta^{\vee}-(2j-i+1)J_0-J_1\bigr)\Bigr) \neq 0.$$
By Lemma \ref{coh-lem-quad} and semicontinuity, $H^0\bigl(D,\OO_D(\phi^*{\zeta^{\vee}}-(2j+1-i)J_0-J_1)\bigr)=0$, so the above is equivalent to
$$K_{n-1-j,1}\bigl (D, \OO_D(\phi^*\zeta^{\vee}-(2j-i+1)J_0-J_1), K_{D}\bigr) \neq 0.$$ By Lemma \ref{coh-lem-quad} and semicontinuity for Koszul cohomology, we then also have
$$K_{n-1-j,1}\bigl(Y,\OO_Y(\phi^*\zeta^{\vee}-(2j-i+1)J_0-J_1), \widetilde{K}\bigr) \neq 0,$$
where, recall that $Y=Z\cup J_0$, with $Z\cdot J_0=x_0$. Consider again the exact sequence (\ref{exseq}) and since
$H^0(J_0,\OO_{J_0}(mJ_0+\phi^*\zeta^{\vee}))=0$ for any $m$, the inclusion map yields isomorphisms
$$ H^0\Bigl(Z, K_{Z}^{\otimes m}\otimes \OO_Y\bigl(\phi^*\zeta^{\vee}+(2m-2j+i-2)x_0-x_1\bigr)\Bigr)
\cong H^0\Bigl(Y,\widetilde{K}^{\otimes m}\otimes \OO_Y\bigl(\phi^*\zeta^{\vee}-(2j-i+1)J_0-J_1\bigr)\Bigr).
$$
valid for all positive integers $m$. Recall the isomorphism $H^0(Y,\widetilde{K}) \cong H^0\bigl(Z,K_{Z}(2x_0)\bigr)$ given by restriction. We define the graded $\mbox{Sym } H^0(Z,K_Z(2x_0))$-module
$$A:=\bigoplus_{q\in \mathbb Z} H^0\Bigl(Z,\OO_Z(\widetilde{K}^{\otimes q}+\phi^*\zeta^{\vee}-(2j-i+1)x_0-x_1)\Bigr),$$
as well as the graded $\mbox{Sym } H^0(Y,\widetilde{K})$-module
$$B:=\bigoplus_{q\in \mathbb Z} H^0\Bigl(Y,\widetilde{K}^{\otimes q}\otimes \OO_Y(\phi^*\zeta^{\vee}-(2j-i+1)J_0-J_1)\Bigr).$$
We then have the following commutative diagram, where the vertical arrows are isomorphisms induced by tensoring the exact sequence (\ref{exseq}):
$$
\xymatrix{
\bigwedge^{n-1-j} H^0\bigl(Z, K_Z(2x_0)\bigr) \otimes A_1
\ar[r] \ar[d] &\bigwedge^{n-2-j} H^0\bigl(Z, K_Z(2x_0)\bigr) \otimes A_2 \ar[d] \\
\bigwedge^{n-1-j} H^0(Y, \widetilde{K}) \otimes B_1 \ar[r]& \bigwedge^{n-2-j} H^0(Y, \widetilde{K}) \otimes B_2,}
$$
Thus it follows
$K_{n-1-j,1}\Bigl(Z,\OO_Z(\phi^*\zeta^{\vee}+(i-2j-2)x_0-x_1), K_Z(2x_0)\Bigr) \neq 0$,
or, equivalently, $$H^0\Bigl(Z,\bigwedge^{n-1-j}M_{K_{Z}(2x_0)} \otimes K_Z\otimes \OO_Z(\phi^*\zeta^{\vee}+(i-2j-2)x_0-x_1\Bigr) \neq 0.$$
We now compute the slope  $\chi\Bigl(\bigwedge^{n-1-j}M_{K_Z}(2x_0) \otimes K_Z \otimes \OO_Z(\phi^*\zeta^{\vee}+(i-2j-2)x_0-x_1\Bigr))=0$. Applying
once more the description given in \cite{beauville-stable} Proposition 2 for the theta divisors of the exterior powers of the vector bundle $M_{K_Z(2x_0)}$, we obtain that either
$$\OO_Z\bigl(\phi^*\zeta+(2j-i)x_0+x_1\bigr)  \in Z_{n-i+j}-Z_{n-1-j},$$
or
$$\OO_Z\bigl(\phi^*\zeta+(2j+2-i)x_0+x_1\bigr)  \in Z_{n+1-i+j}-Z_{n-2-j},$$
which establishes the claim.
\end{proof}

\vskip 3pt
We now complete the proof of the Prym--Green Conjecture for odd genus.
\begin{thm}\label{finalcheck}
Set $g=2n+1$ and $\ell\geq 2$. Then for a general element $[C,\tau]\in \cR_{g,\ell}$ one has $K_{n-3,2}(C,K_C\otimes \tau)=0$.
\end{thm}
\begin{proof}
Using the inductive argument from Proposition \ref{coh-lem-quad},  it suffices to prove the base case of the induction, that is,   show  that if $D\in |J_0+f_r|$ is general and $0\leq j\leq g-1$, then
$$\OO_D\bigl(\phi^*\zeta+(2j-g+2)J_0+J_1\bigr) \notin D_{n-1-j}-D_{n-1-j}, $$
for any $0 \leq j \leq g-1$. Suppose this is not the case, which forces  $j=n-1$ and then
$$\OO_D\bigl(\phi^*(\zeta\otimes \eta^{\vee})\bigr)\cong  \OO_{D}(\phi^* \zeta-J_0+J_1) \cong  \mathcal{O}_{D}.$$ Since $H^0\bigl(X, \phi^*(\zeta\otimes \eta^{\vee})\bigr)=0$, this implies $H^1\bigl(X,\phi^*(\zeta\otimes \eta^{\vee}\otimes \OO_X(-J_0-f_r)\bigr) \neq 0$. Observe $H^2(X,\phi^*(\zeta\otimes \eta^{\vee})\otimes \OO_X(-J_0-f_r)\bigr)=0$ by Serre duality. Taking the cohomology exact sequence associated to
$$0 \longrightarrow  \phi^*(\zeta\otimes \eta^{\vee})\otimes \OO_X(-J_0-f_r) \longrightarrow  \phi^*(\zeta\otimes \eta^{\vee})\otimes \OO_X(-f_r) \longrightarrow  \OO_{J_0}\bigl(\phi^*(\zeta\otimes \eta^{\vee})-f_r\bigr) \longrightarrow  0$$
and using the Leray spectral sequence, we immediately get $H^1\bigl(X,\phi^*(\zeta\otimes \eta^{\vee})\otimes \OO_X(-J_0-f_r)\bigr)=0$, which is a contradiction.
\end{proof}

\end{document}